\documentclass[12pt]{amsart}
\usepackage{epsfig,psfrag,verbatim,amssymb,amsmath,amsfonts}
\newtheorem{theorem}{Theorem}
\newtheorem{lemma}[theorem]{Lemma}
\newtheorem{prop}[theorem]{Proposition}

\newtheorem{ex}{Example}
\newtheorem{rem}{Remark}

\newtheorem{problem}{Open Problem}

\newcommand{\Y}{\mathcal{Y}}
\newcommand{\X}{\mathcal{X}}
\newcommand{\F}{\mathcal{F}}
\newcommand{\I}{\mathcal{I}}

\newcommand{\R}{\mathbb{R}}

\newcommand{\Z}{\mathbb{Z}}
\newcommand{\N}{\mathbb{N}}

\newcommand{\C}{\mathbb{C}}

\newcommand{\la} {\lambda}
\newcommand{\si}{\sigma}

\newcommand{\ph}{\varphi}
\newcommand{\om}{\omega}

\newcommand{\Om}{\Omega}

\newcommand{\eps}{\varepsilon}
\newcommand{\won}{{\boldsymbol 1}}

\begin{document}

\setcounter{page}{1}

\title[Interpolation percolation
]{Interpolation percolation}

\author{Martin P.W.\ Zerner} 

\thanks{\textit{Acknowledgment:} The author was partially supported by the German-Israeli
Foundation, grant no.\ 974-152.6/2007.}

\begin{abstract} 
Let $\mathcal X\subset \R$ be countably infinite and let $(\mathcal Y_x)_{x\in \mathcal X}$ be an 
independent family of stationary random sets $\mathcal Y_x\subseteq\R$, e.g.\ homogeneous Poisson point  processes $\Y_x$ on $\R$. 
We 
give criteria for the a.s.\ existence of various ``regular'' 
functions $f:\R\to\R$ with the property that
$f(x)\in\Y_x$ for all $x\in \X$. Several open questions are posed.
\end{abstract}
\maketitle

\section{Introduction}
In classical discrete percolation \cite{Gr99} one randomly and independently deletes  edges or vertices from a graph and considers the properties of the connected components of the remaining graph $V$. In standard continuum percolation \cite{MR96}, which is concerned with  Boolean models,
one removes balls, whose centers form a random point process, from space and again investigates 
the connectivity properties of the remaining part $V$ of space or, somewhat more commonly, of its complement $V^c$. In both cases, 
each component deleted from the underlying medium has, in some sense, a strictly positive volume. Moreover,   any bounded region of the space intersects a.s.\ only  a finite number of these components  (see e.g.\ \cite[Proposition 7.4]{MR96} for continuum percolation).

In contrast, in fractal percolation \cite[Chapter 13.4]{Gr99} and continuum fractal percolation \cite{Za84} \cite[Chapter 8.1]{MR96} one may delete from any bounded open region  countably many components, each of which has a strictly positive volume. For example, \cite{Za84} deals with the Hausdorff dimension of the set 
$V=\R^n\backslash\bigcup_{i\in\N}\Gamma_i$,
where $\Gamma_1,\Gamma_2,\ldots$ are independent random open sets in $\R^n$, 
the so-called cutouts, 
e.g.\ scaled Boolean models.

In the present paper we introduce another continuum percolation model, in $\R^2$, in which countably many components may be removed from any bounded region. However, in contrast to the previous models, 
these components 
are null sets in $\R^2$. In particular, the remaining set $V\subset \R^2$ has full Lebesgue measure.  
The cutouts are arranged in such a way that $V$ exhibits various phase transitions.
The precise model is the following.

Fix a complete  probability space $(\Om,\F,P)$  and 
a countably infinite set $\X\subset \R$. For all $x\in\X$ let $\Y_x\subseteq \R$ be a random closed set, i.e.\ 
$\Y_x$ is a 
random variable on $(\Om,\F,P)$ with values in the space of closed subsets of $\R$ equipped with the Borel $\si$-algebra generated by the so-called Fell topology, see \cite[Chapter 1.1.1]{Mo05} for details.
 Throughout we suppose that 
\begin{equation}\label{hab}
P[\Y_x=\emptyset]=0\quad\mbox{for all $x\in\X$.}
\end{equation} 
Another common assumption will be that 
\[\mbox{(IND)}\hspace*{10mm}\begin{array}{l}\mbox{$(\Y_x)_{x\in\X}$ is independent,}
\end{array}\hspace*{100mm}\]
which is defined as usual,  see \cite[Definition 1.1.18]{Mo05}.
We will also need to assume some translation invariance of the sets $\Y_x$ as described in \cite[Chapter 1.4.1]{Mo05}. One such possible assumption is that 
\[\mbox{(STAT)}\hspace*{10mm}\begin{array}{l}\mbox{$\Y_x$ is stationary for all $x\in\X$, i.e.\  $\Y_x$ has} \\
\mbox{for all $a\in\R$ the same distribution as $\Y_x+a$.}
\end{array}\hspace*{50mm}\]
This assumption is stronger than the hypothesis that 
\[\mbox{(1STAT)}\hspace*{10mm}\begin{array}{l}\mbox{$\Y_x$ is  first-order stationary for all $x\in\X$, i.e.\ 
for all $a,b,c\in\R$,}\\
\mbox{with $a\le b$, $P\left[\Y_x\cap[a+c,b+c]\ne\emptyset\right]=P\left[\Y_x\cap[a,b]\ne\emptyset\right]$. }
\end{array}\hspace*{30mm}\]
Our main example, which satisfies all of the above conditions, is the following:
\[\mbox{(PPP)}\hspace*{10mm}\begin{array}{l}\mbox{$(\Y_x)_{x\in\X}$ are independent homogeneous Poisson point processes}\\
\mbox{on $\R$ with intensities $\la_x>0$, where $(\la_x)_{x\in\X}$ is fixed.}\end{array}\hspace*{30mm}\]
For a renewal process with interarrival times which are not exponentially distributed, like under (PPP), but Weibull distributed, see Example \ref{bb}. Another example, which fulfills (IND) and (STAT) and is periodic, is 
\begin{equation}\label{ano}\begin{array}{l}\mbox{$\Y_x=\la_x(K_x+\Z+U_x),$\quad where $\la_x\ne 0$ and all $K_x\subset \R\ (x\in\X)$}\\
\mbox{are compact and $(U_x)_{x\in\X}$ is i.i.d.\ with $U_x\sim{\rm Unif}[0,1]$.}
\end{array}
\end{equation}

We now remove for each $x\in\X$  the set $\{x\}\times\Y_x^c$ from $\R^2$, where $\Y_x^c=\R\backslash\Y_x$, and  investigate the remaining set
\[V:=\R^2\backslash\bigcup_{x\in\X}\left(\{x\}\times \Y_x^c\right)=\left\{(x,y)\in\R^2\mid x\in\X\Rightarrow y\in\Y_x\right\}.
\]
 We thus ``perforate" $\R^2$ by cutting out random subsets of parallel vertical lines  to obtain a ``vertically dependent" random set $V\subset\R^2$. 
For a similar discrete percolation model with vertical dependence see \cite[Section 1.6]{Gr09}.

What are the topological properties of $V$?
It is easy to see  that $V$ is always connected, see Proposition \ref{conn}. 
However, 
whether $V$ is  path-connected or not depends on the parameters. (Recall that $V$ is path-connected if and only if for all $u,v\in V$ there is a continuous function $f:[0,1]\to V$ with $f(0)=u$ and $f(1)=v$.) 
\begin{theorem}\label{AA}
Assume (PPP) and let $\X$ be bounded. Then $V$ is a.s.\ path-connected if
\begin{equation}\label{wm}
\forall \eps>0\quad \sum_{x\in\X}e^{-\la_x\eps}<\infty,
\end{equation}
and a.s.\ not path-connected otherwise.
\end{theorem}
\begin{ex}\label{ex2}{\rm Assume (PPP), suppose $\X$ is bounded and let  $(x_n)_{n\in\N}$ enumerate $\X$. 
Then 
 $V$ is a.s.\ path-connected if $\log n=o(\la_{x_n})$ and a.s.\ not path-connected if $\la_{x_n}=O(\log n)$, cf.\ Example \ref{ex1}.}
\end{ex}

Note that (\ref{wm}) depends only on the intensities $\la_x$, counted with multiplicities, but not on $\X$ itself.

Theorem \ref{AA} will be generalized in Theorem \ref{cont}. There it will also be shown that $V$ is a.s.\ path-connected if and only if  there
is a continuous function $f:\R\to\R$ whose graph $\mbox{graph}(f):=\{(x,f(x))\mid x\in\R\}$ is contained in $V$.  This brings up the question under which conditions there are functions $f:\R\to\R$ which have other regularity properties than continuity and  
which belong to
\[\I:=\left\{f:\R\to\R\mid{\rm graph}(f)\subseteq V\right\}=\left\{f:\R\to\R\mid \forall x\in\X: f(x)\in \mathcal Y_x\right\}.\]
The elements of $\I$ in some sense interpolate the sets $\Y_x,\ x\in\X$, see Figure \ref{333} for examples. For this reason we suggest the name \textit{interpolation percolation} for this model.
\begin{figure}[t]\label{1}
\epsfig{figure=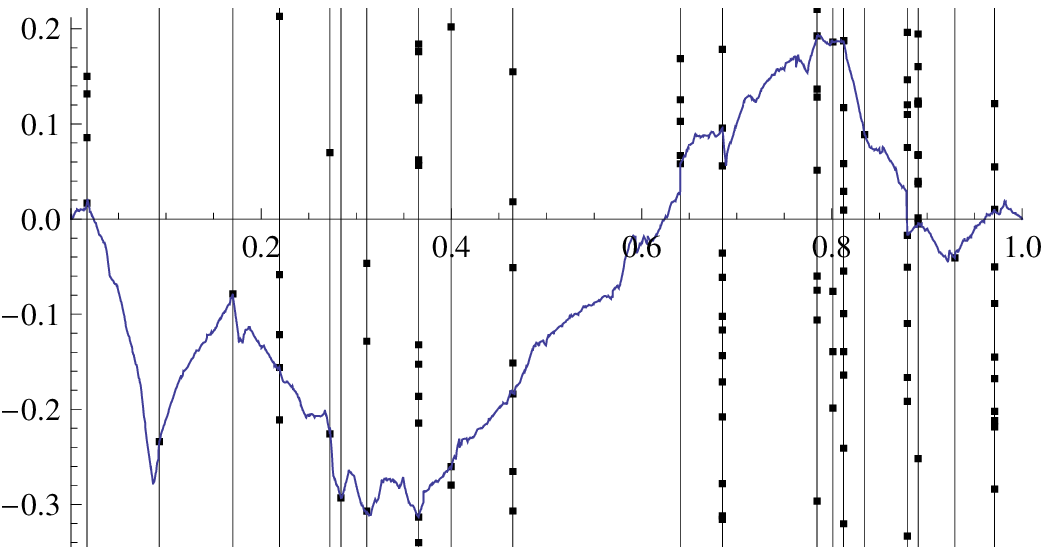, width=270pt}\hspace*{3mm} 
\epsfig{figure=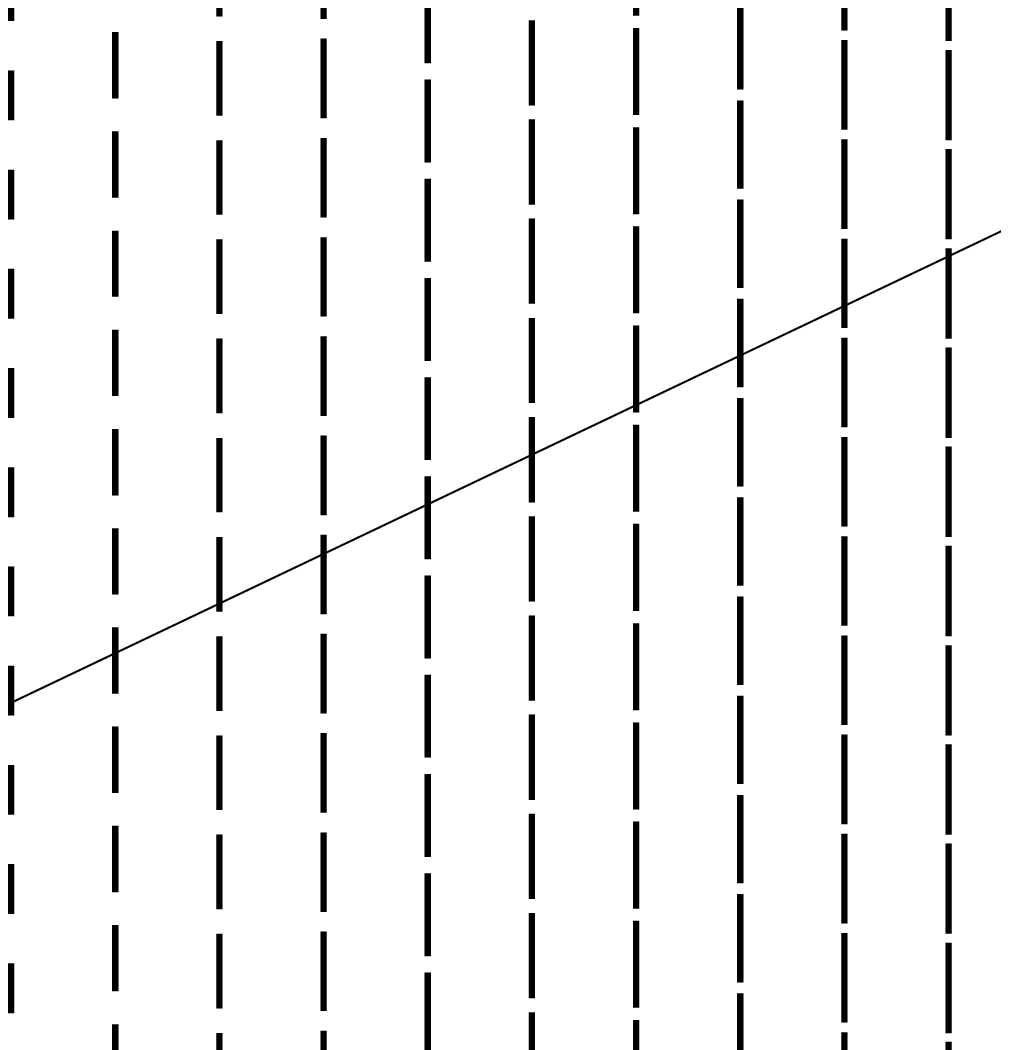,width=140pt} 
\caption{\footnotesize (a) In the left figure $x_n, n\in\N$, are independent and uniformly distributed on [0,1], $\X=\{x_n:\ n\in \N\}$, and $\Y_{x_n}$, $n\in\N$,  are independent Poisson point processes of intensity $\la_{x_n}=n$, independent also of $(x_n)_{n\in\N}$. The figure shows $\Y_{x_1},\ldots,\Y_{x_{20}}$ and  an interpolating continuous function $f\in\I$ with $f(0)=f(1)=0$. (b) In the right figure $\X=\{2,3,\ldots\}$ and $\Y_n=[1/n,1]+\Z+U_n$, where $U_n, n\ge 2$, are independent and uniformly distributed on [0,1].
The figure shows  $\Y_{2},\ldots,\Y_{11}$ and an interpolating line, see also Open Problem \ref{speed}.}\label{333}
\end{figure}
In Section 2 we derive some conditions under which the following subsets of $\Om$, which describe the existence of interpolating functions with various regularity properties, occur or don't occur. 
\[\begin{array}{rcll}
B&:=&\left\{\exists f\in\I:\ f\mbox{ is bounded}\right\}&\mbox{(see Prop.\ \ref{bou}),}\vspace*{1mm}\\
C&:=&\left\{
\exists f\in\I:\ f\mbox{ is continuous}\right\}&\mbox{(see Th.\ \ref{cont}),}\vspace*{1mm}\\
M&:=&\left\{
\exists f\in\I:\ f\mbox{ is increasing\footnotemark\ and bounded}\right\}&\mbox{(see Th.\ \ref{mono}),}\vspace*{1mm}\\
BV&:=&\left\{
\exists f\in\I:\ f\mbox{ is of bounded variation}\right\}&\mbox{(see Th.\ \ref{bv}),}\vspace*{1mm}\\
L_K&:=&\multicolumn{2}{l}{\left\{
\exists f\in\I:\ f\mbox{ is Lipschitz continuous with Lipschitz constant $K$}\right\},}\\
&&\mbox{ where $K>0$,}&\mbox{(see Th.\ \ref{li}),}\vspace*{1mm}\\
P_m&:=&\left\{
\exists f\in\I:\ f\mbox{ is a polynomial of degree $m$}\right\},\\
&&\mbox{ where $m\in\N_0=\{0,1,2,\ldots\}$,}&\mbox{(see Prop.\ \ref{po}, Th.\ \ref{shepp}),}\vspace*{1mm}\\
A&:=&\left\{
\exists f\in\I:\ f\mbox{ is real analytic}\right\}&\mbox{(see Prop.\ \ref{ana}).}
\end{array}
\]
\footnotetext{Here and in the following a function $f$ is called increasing if and only if $f(s)\le f(t)$ for all $s<t$. Similarly, $f$ is called decreasing if $-f$ is increasing.}
It will follow from the completeness of $(\Om,\F,P)$ that these sets are  events, i.e.\ elements of $\F$. Our results for $B, P_m$ and $A$ are not difficult.

\begin{rem}\label{0E}{\rm {\bf  (0-1-law)} Note that every event defined above, call it $G$, is invariant under vertical shifts of $V$, i.e.\ $G$  occurs if and only if it occurs after replacing $V$ by any $V+(0,a)$, $a\in\R$. Therefore, under suitable ergodicity assumptions, $P[G]\in\{0,1\}$.
This is the case if (PPP) holds. In general, e.g.\ in case (\ref{ano}), this need not be true, see Remark \ref{no0}.}
\end{rem}
\begin{rem}\label{mm}{\rm {\bf  (Monotonicity)}
There is an obvious monotonicity property:
If $\X$ is replaced by  $\X'\subseteq \X$ and  $(\Y_x)_{x\in\X}$ by $(\Y'_{x'})_{x'\in\X'}$ with $\Y'_{x'}\supseteq \Y_{x'}$ for all $x'\in\X'$ then $V$ and $\I$ and all the events defined above and their probabilities increase. 
}
\end{rem}

\begin{rem}\label{clo}{\rm {\bf (Closedness)} There are various notions of random sets. Random closed sets
seem to be the best studied ones, see \cite{Mo05} and Chapter 1.2.5 therein for a discussion of non-closed random sets. For this reason 
we assume the sets $\Y_x$ to be closed even though this assumption does not seem to be essential for our results. An alternative, but seemingly less common notion of stationary random and not necessarily closed sets is described e.g.\ in \cite[Chapter 8]{JKO94}. 
}
\end{rem} 

\begin{rem}\label{rel}{\rm {\bf (Further connections to other models)} 
This model is related to various other models in probability.

(a) (Brownian motion) Our construction of the continuous functions in the proof of Theorem \ref{AA} has been inspired by Paul L\'evy's method of constructing Brownian motion. In Example \ref{bb} we shall even choose  $\X\subseteq (0,1]$ and $(\Y_x)_{x\in\X}$ in a non-trivial way and use this method to define  a Brownian motion $(B_x)_{0\le x\le 1}$  on the same probability space $(\Om,\F,P)$ such that a.s.\ $(B_x)_{x\in\R}\in\I$. (Here we let $B_x:=0$ for $x\notin(0,1]$ to extend $B_\cdot(\om)$ to a function defined on $\R$.)

(b) (Lipschitz and directed percolation) In \cite{DDGHS10} random Lipschitz functions $F:\Z^{d-1}\to\Z$ are constructed such that, for every $x\in\Z^{d-1}$, the site
$(x,F(x))$ is open in a site percolation process on $\Z^d$. 
The case $d=2$ is easy since it is closely related to  oriented site percolation on $\Z^2$. 
However, if we let $d=2$, denote the Lipschitz constant by $L$ and let the parameter $p_L$ of the site percolation process 
depend on $L$ such that $Lp_L \to \la\in(0,\infty)$ as $L\to\infty$ then we obtain after applying the scaling $(x,y)\mapsto (x,y/L)$ in the limit
$L\to\infty$ 
the problem of studying the event $L_1$ under the assumption (PPP) with $\X=\Z$ and $\la_x=\la$ for all $x\in\X$. Theorem \ref{li} deals with
this case and is proved using oriented percolation. 

(c) (First-passage percolation) Our computation of  $P[BV]$ in Theorem \ref{bv} applies a method used for the study of first-passage percolation on  spherically symmetric  trees in \cite{PemPer94}. 

(d) (Intersections of stationary random sets)
Note that 
\begin{equation}\label{int}
P_0=\Big\{
\bigcap_{x\in\X} \Y_x
\ne\emptyset \Big\}.
\end{equation}
Such intersections of stationary random sets (or the unions of their complements) have been investigated e.g.\ in \cite{Sh72a}, \cite{Sh72b}, \cite{KP91} and \cite{JS08}. Studying the events $P_m\supseteq P_0$, $m\ge 1,$ yields natural variations of these problems, see e.g.\ Open Problem \ref{speed}.

(e) (Poisson matching)  In 2-color Poisson matching \cite{HPPS09} one  is concerned with matching the points of one Poisson point process to the points of another such process in a translation-invariant way. In our model with assumption (PPP), $\infty$-color Poisson matching would correspond to
choosing infinitely many $f\in\I$ in a translation-invariant way.
This  will be made more precise in Open Problem \ref{rm}.
}
\end{rem}

\section{Results, proofs, and open problems}\label{inter}
Some of our results will be phrased in terms of the following random variables.
We denote for $x\in\X$ and $z\in\R$ by
\begin{eqnarray}
D^+_x(z)&:=&\inf\{y-z:\ y\ge z,\, y\in \mathcal Y_x\},\nonumber\\
D^-_x(z)&:=&\inf\{z-y:\ y\le z,\, y\in \mathcal Y_x\}, \quad\mbox{and}\nonumber \\
D_x(z)&:=&\inf\{|y-z|:\ y\in\Y_x\}=\min\{D^+_x(z),D^-_x(z)\}\nonumber
\end{eqnarray}
the distance of $z\in\R$ from $\Y_x\cap[z,\infty)$, $\Y_x\cap(-\infty,z]$, and $\Y_x$, respectively. 
\begin{lemma}\label{2} The probability that
$D_x(z)$ is finite for all $x\in\X$ and  all $z\in\R$ is 1. 
If (1STAT) holds then the distributions of  $D_x(z), D_x^+(z)$ and $D_x^-(z)$
do not depend on $z$ and $2D_x(0)$ has the same distribution as $D_x^+(0)$ and $D_x^-(0).$
\end{lemma}
\begin{proof} The first statement follows from $\{\forall z\in\R: D_x(z)<\infty\}=\{\Y_x\ne\emptyset\}$ and (\ref{hab}).
If (1STAT) holds then
\[
P[D_x(z)\le t]=P[\Y_x\cap[z-t,z+t]\ne \emptyset]=P[\Y_x\cap[-t,t]\ne\emptyset]=P[D_x(0)\le t].
\]
Similarly, $P[D_x^+(z)\le t]=P[D_x^+(0)\le t]=P[D_x(0)\le t/2]$.
Analogous statements hold for $D_x^-$.
\end{proof}
\begin{ex}\label{paul}{\rm Under assumption (PPP) all three distances $D^+_x(z),D^-_x(z)$ and $D_x(z)$ are exponentially distributed with respective 
parameters $\la_x$, $\la_x$, and  $2\la_x$.}
\end{ex}
For comparison with Theorem \ref{cont} about continuous functions and as a warm-up we first consider bounded, but not necessarily continuous functions.
\begin{prop}\label{bou}{\rm \bf (Bounded functions)}
Assume (IND). If 
\begin{equation}\label{this}
\sup_{x\in\X}D_x(0)<\infty\quad\mbox{a.s.}
\end{equation}
then $P[B]=1$, otherwise $P[B]=0$.
\end{prop}
\begin{proof}
If (\ref{this}) holds then a bounded function $f\in\I$ can be  defined by setting $f(x)=0$ for $x\notin \X$ and choosing
$f(x)\in\Y_x$ with $|f(x)|=D_x(0)$ for all $x\in\X$. For the converse we note that any $f\in\I$ must satisfy $\|f\|_\infty\ge\sup_{x\in\X}D_x(0)$ and apply Kolmogorov's zero-one law.
\end{proof}
\begin{rem}\label{re1}{\rm By the  Borel Cantelli lemma (\ref{this}) is equivalent to  
\begin{equation} \label{sup}
\exists\ t<\infty\quad  \sum_{x\in\X}P[D_x(0)>t]<\infty.
\end{equation}
Note that (\ref{this}) and (\ref{sup}) do not depend on $\X$ but only on the distributions of the random variables $D_x(0)$,\ $x\in\X$.
}\end{rem}
\begin{ex}\label{ex1}{\rm Suppose (PPP) holds and $(x_n)_{n\in\N}$ enumerates $\X$. Then 
 $P[B]=1$ if $\log n=O(\la_{x_n})$ and $P[B]=0$ if $\la_{x_n}=o(\log n)$, cf.\ Example \ref{ex2}.}
\end{ex}
Theorem \ref{AA} immediately follows from the following result.
\begin{theorem}\label{cont}{\rm \bf (Continuous functions)} Assume (IND) and (1STAT).
If \begin{equation}\label{doc}
\forall m\in\N\quad \forall \eps>0\quad  \sum_{x\in\X, |x|\le m}P[D_x(0)>\eps]<\infty
\end{equation}
then  $P[C]=P[V\mbox{ is path-connected}]=1.$ Otherwise $P[C]=P[V\mbox{ is path-connected}]=0.$
\end{theorem}
\begin{rem}{\rm Suppose $(x_n)_{n\in\N}$ enumerates $\X$. 
Then, by the  Borel Cantelli lemma, (\ref{doc}) is equivalent to 
\begin{equation}\label{red}
\forall m\in\N\ \lim_{n\to\infty}D_{x_n}(0)\won_{[-m,m]}(x_n)=0\quad\mbox{a.s..}
\end{equation}
Compare (\ref{doc}) to (\ref{sup}) and (\ref{red}) to (\ref{this}). Also note
 that if $\X$ is bounded then (\ref{doc}) does not depend on $\X$ itself but only on the distributions of the random variables $D_x(0)$,\ $x\in\X$.  
}
\end{rem}
\noindent\textit{Proof of Theorem \ref{cont}.}
Assume (\ref{doc}).  Then for all $y\in \R$ and all $\eps>0$ the set $\X(y,\eps):=\{x\in\X\mid D_x(y)\ge \eps/3\}$ is a.s.\ locally finite due to the Borel Cantelli lemma and Lemma \ref{2}. Therefore, for all $K,\eps>0$ the
set 
\[
\X_K(\eps):=\{x\in\X\mid \exists y\in[-K,K]: D_x(y)\ge \eps\}\\
\]
is a.s.\ locally finite as well since it is contained in the union of the sets $\X(y,\eps)$ with $y\in[-K,K]\cap (\eps/3) \Z$.
This together with (\ref{hab}) and the monotonicity of $\X_K(\eps)$ in $K$ and in $\eps$ 
implies the existence of a set $\Om'\subseteq \Om$ of full $P$-measure on which
\begin{equation}\label{ejp}
\forall K, \eps>0: \X_K(\eps)\ \mbox{is locally finite\quad and}\quad\forall x\in\X: \Y_x\ne \emptyset.
\end{equation} 
For the proof of the first statement of the theorem it suffices to show that on $\Om'$ there is for all $a_0, a_1, b_0, b_1\in\R$ with $a_0<a_1$ some continuous function $f=f_{a_0,a_1,b_0,b_1}:[a_0,a_1]\to\R$ with 
\begin{equation}\label{ggf}
f(a_0)=b_0,\quad f(a_1)=b_1\quad\mbox{and}\quad f(x)\in\Y_x\quad\mbox{for all $x\in\X\cap]a_0,a_1[$.} 
\end{equation} 
Indeed, this immediately shows that on $\Om'$ any two points in $V$ with differing first coordinates $a_0$ and $a_1$ can be connected by a continuous path inside $V$. Points in $V$ whose first coordinates coincide can be connected by the concatenation of two such paths.
Similarly,
$P[C]=1$ follows by concatenating the functions $f_{a_n,a_{n+1},0,0}$, where $(a_n)_{n\in\Z}$ is a strictly increasing double sided sequence in $\R\backslash \X$ with $a_n\to\infty$ and $a_{-n}\to-\infty
$ as $n\to\infty$. 

Therefore, let $a_0, a_1, b_0, b_1\in\R$ with $a_0<a_1$.
If $\X\cap]a_0, a_1[$ is finite then the existence of a continuous function $f_{a_0,a_1,b_0,b_1}$ satisfying (\ref{ggf}) is obvious. Now assume that $\X\cap]a_0, a_1[$ is infinite. Also fix a realization in $\Om'$. 
By (\ref{ejp}), 
\begin{equation}\label{M}
K:=\left(\sup\left\{D_x(0): x\in\X_1(1)\cap]a_0,a_1[\right\}\vee |b_0|\vee|b_1|\right)+2
\end{equation}
is finite. Here $\sup\emptyset:=0$. The sets $A_n, n\in\N_0$, recursively defined by  
\[A_{n}:=\left(\X_K(2^{-n})\cap]a_0,a_1[\right)\backslash \bigcup_{k=0}^{n-1}A_k\qquad(n\in\N_0),\]
are finite due to (\ref{ejp}), pairwise disjoint and satisfy
\begin{equation}\label{cup}
\bigcup_{n\ge 0}A_n=\bigcup_{n\ge 0}\X_K(2^{-n})\cap]a_0,a_1[= 
\left(\X\cap]a_0,a_1[\right)\backslash\{x\in\X : [-K,K]\subseteq\Y_x\}.
\end{equation}
We shall obtain the desired function $f=f_{a_0,a_1,b_0,b_1}$ as uniform limit of a sequence  $(f_n)_{n\ge 0}$ of continuous functions which satisfy for all $n\ge 0$,
\begin{equation}\label{von}
f_{n}:[a_0,a_1]\to\left(K-2+\sum_{k=0}^{n-1}2^{-k}\right)\left[-1,1\right]\subseteq [-K,K],
\end{equation}
\begin{equation}\label{ab}
f_n(a_0)=b_0,\quad f_n(a_1)=b_1,
\end{equation}
\begin{equation}\label{ag}
f_n(x)\in\Y_x\quad\mbox{for all $x\in A_n$,}
\end{equation}
\begin{equation}\label{ag2}
f_n(x)=f_{n-1}(x)\quad\mbox{for all  $A_0\cup\ldots\cup A_{n-1}$ if $n\ge 1$,\quad\mbox{and}}
\end{equation}
\begin{equation}\label{ofer}
\|f_{n}-f_{n-1}\|_\infty\le 2^{-n+1}\quad\mbox{if $n\ge 1$}.
\end{equation}
We construct this sequence recursively.
For $n=0$ set $f_0(a_0)=b_0$, $f_0(a_1)=b_1$ and choose $f_0(x)\in \Y_x$
for all $x\in
A_0$ such that $|f_0(x)|=D_x(0)$. 
This defines $f_0(x)$ for finitely many $x$. By linearly interpolating in between these values we get a continuous piecewise linear function $f_0$ satisfying (\ref{von})--(\ref{ofer}) with $n=0$. 

Now let $n\ge 0$ and assume that we have already constructed a continuous function $f_n$ which fulfils (\ref{von})--(\ref{ofer}).
We then set $f_{n+1}(x)=f_n(x)$ for all $x\in \{a_0,a_1\}\cup A_0\cup\ldots\cup A_n$.
For all $x\in A_{n+1}$, which again is a finite set, we choose $f_{n+1}(x)$ as some element of $\Y_x$ with distance $D_x(f_n(x))$ from $f_n(x)$
and again interpolate linearly to obtain a piecewise linear continuous function $f_{n+1}$, see Figure  \ref{c2} (a).
\begin{figure}[t]
 \psfrag{a}{$f_n$}
 \psfrag{b}{$f_{n+1}$}
\psfrag{c}{$x\in A_n$}
\psfrag{d}{$x\in A_{n+1}$}
\epsfig{figure=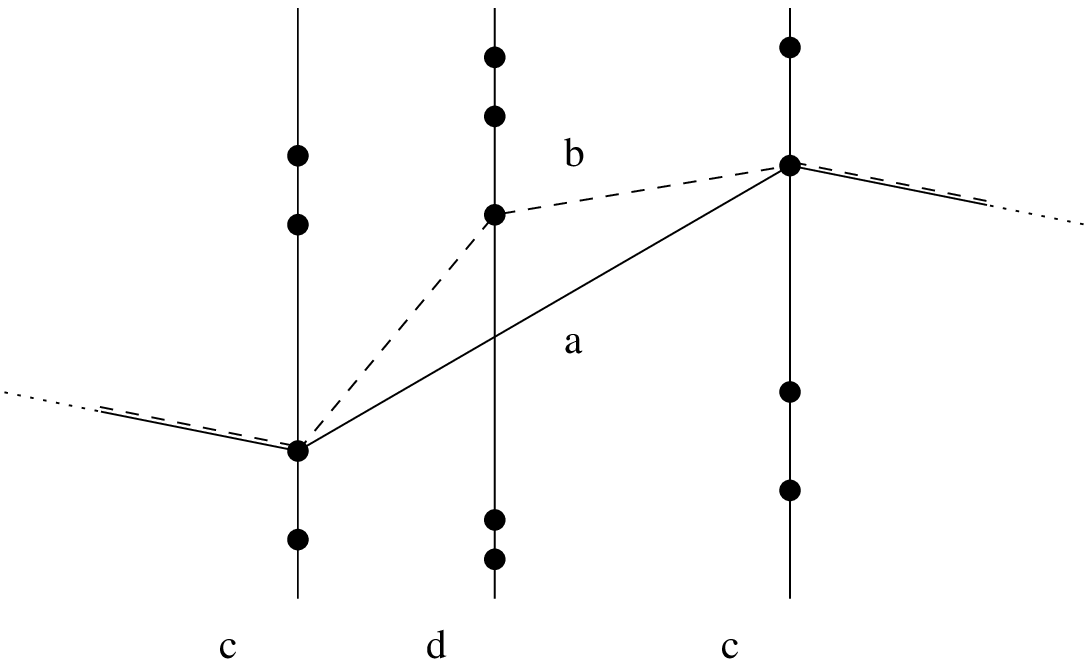,height=135pt}\hspace*{5mm}
 \psfrag{e}{$\eps$}
 \psfrag{bb}{$\xi$}
\psfrag{xt}{$f(\tau_f)$}
\psfrag{xs}{$f(\si_f)$}
\epsfig{figure=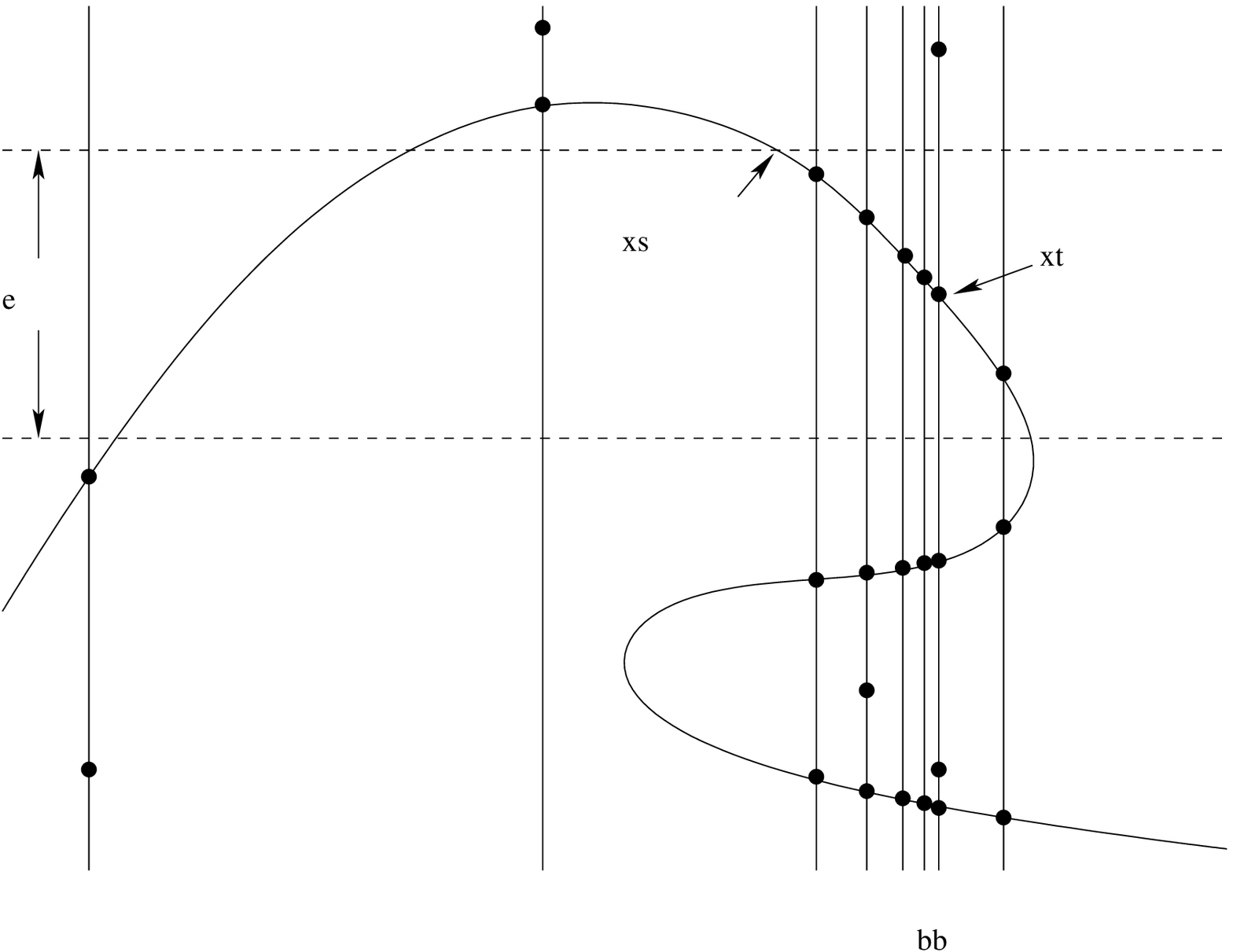,height=135pt}
\caption{\footnotesize For the proof of Theorem \ref{cont} in case (PPP).
}\label{c2}
\end{figure} 
 Obviously, this function satisfies (\ref{ab})--(\ref{ag2}) with $n+1$ instead of $n$. To see that the same is true for (\ref{von}) and (\ref{ofer}) we note that 
for all $x\in A_{n+1}$, by (\ref{von}),  $f_n(x)\in[-K,K]$.
Therefore, $D_x(f_n(x))\le 2^{-n}$ since $x\notin \X_K(2^{-n})$. Hence
$|f_{n+1}(x)-f_{n}(x)|\le 2^{-n}$. Consequently, since we interpolated linearly,
\begin{equation}\label{max}
\|f_n-f_{n+1}\|_\infty=\sup \left\{|f_{n+1}(x)-f_{n}(x)|:\  x\in A_{n+1}\right\}\le 2^{-n},
\end{equation}
i.e.\ (\ref{ofer}) holds with $n+1$ instead of $n$. This together with (\ref{von}) implies that (\ref{von}) also holds for  $n+1$ instead of $n$.

Having finished the construction of $(f_n)_{n\ge 0}$, we see that it converges uniformly on $[a_0,a_1]$ due to (\ref{ofer})
to some function $f$. Since all the functions $f_n$ are continuous $f$ is continuous as well.
It also fulfils (\ref{ggf}) due to (\ref{cup})--(\ref{ag2}).\vspace*{2mm}
 
For the proof of the converse assume that there are  $m\in\N$ and $\eps>0$  such that 
$\sum_{x\in\X,|x|\le m}P[D_x(0)>\eps]=\infty.$
Since $[-m,m]$ is compact there is  $\xi\in[-m,m]$ such that for all $\delta>0$ the sum of the probabilities $P[D_x(0)>\eps]$ with $x\in \X\cap (\xi-\delta,\xi+\delta)$ diverges.
Without loss of generality we may assume that one can choose $\xi\in[-m,m]$ such that even
\begin{equation}\label{xi}
\sum_{x\in \X\cap(\xi-\delta,\xi)}P[D_x(0)>\eps]=\infty\quad\mbox{for all $\delta>0$.}
\end{equation}
(Otherwise replace $\X$ by $-\X$.) Denote by $C_2$ the event that there is a continuous function
$f=(f_1,f_2):[0,1]\to\R^2$ with $f_1(0)<\xi<f_1(1)$ whose graph is contained in $V$.  Note that $C\subseteq C_2$ and $\{V\mbox{ is path-connected}\}\subseteq C_2$. Therefore, it suffices to show $P[C_2]=0$.  For any such function $f$
define $\tau_f:=\min\{t\in [0,1]\mid f_1(t)=\xi\}.$
Since $f_1$ is continuous 
$\tau_f$ is well-defined. Next we set
$
\si_f:=\sup\left\{t\in[0,\tau_f]:\ |f_2(t)-f_2(\tau_f)|\ge \eps/2\right\}.
$
Since $f_2$ is continuous as well we have $0\le\si_f< \tau_f$, see Figure \ref{c2} (b).
Therefore, $f_1(\si_f)<\xi$ by definition of $\tau_f$.
Consequently,
\begin{eqnarray*}
C_2=\bigcup_{j\in\Z,k\in\N}\big\{\exists f\in C\left([0,1],\R^2\right)&:& f_1(0)<\xi<f_1(1),\ {\rm graph}(f)\subset V,\\
&&
|f_2(\tau_f)-j\eps|\le\eps/2,\ f_1(\si_f)<\xi-1/k\big\}.
\end{eqnarray*}
If  $f_1(\si_f)<\xi-1/k$ then there is by the intermediate value theorem for all $x\in\X\cap(\xi-1/k,\xi)$ some $t\in(\si_f,\tau_f)$ with $f_1(t)=x$. For such $t$ we have on the one hand by the definition of $\si_f$ that $|f_2(t)-f_2(\tau_f)|<\eps/2$  
and on the other hand $f_2(t)\in\Y_x$ since ${\rm graph}(f) \subset V$.
Therefore, by the triangle inequality,
\[
C_2\subseteq\bigcup_{j\in\Z,k\in\N}C_{j,k},\quad\mbox{where}\quad C_{j,k}:=\left\{\forall x\in\X\cap(\xi-1/k,\xi):\ D_x(j\eps)\le \eps\right\}\in\F.
\]
Consequently, it suffices to show that $P[C_{j,k}]=0$ for all $j\in\Z,k\in\N$. However, by (IND), (1STAT)
and Lemma \ref{2},
\[\hspace*{29mm}P[C_{j,k}]
=\prod_{x\in\X\cap(\xi-1/k,\xi)}
\left(1-P[D_x(0)> \eps]\right)\stackrel{(\ref{xi})}{=}0.\hspace*{29mm}\Box\]
Whether $V$ is path-connected or not thus depends on the parameters of the model.
One may wonder whether the same is true for the connectedness of $V$. This is not the case. $V$ is always connected as the following non-probabilistic statement shows when applied to $U=V$ and $X=\R\backslash\X$.
\begin{prop}{\bf (Connectedness)} \label{conn}
Let $U \subseteq \R^2$ with projection $\pi[U]=\R$ onto the first coordinate
and let $X\subseteq \R$ be dense in $\R$ with $X\times\R\subseteq U$. Then $U$ is connected.
\end{prop}
\begin{proof}
Assume that $U$ is not connected. Then there are non-empty open sets $\mathcal O_1, \mathcal O_2\subseteq \R^2$ such that 
$U\cap \mathcal O_1$ and $U\cap \mathcal O_2$ partition $U$.
 Since $\R=\pi[U]=\pi[\mathcal O_1]\cup\pi[\mathcal O_1]$ is connected and the 
sets $\pi[\mathcal O_1]$ and $\pi[\mathcal O_2]$ are both non-empty and open, the set $\pi[\mathcal O_1]\cap \pi[\mathcal O_2]$ is not empty either.
Since it is also open and $X$ is dense in $\R$ there is  
 $x\in X\cap \pi[\mathcal O_1]\cap\pi[\mathcal O_2]$. For any $i=1,2$, $U_i:=(\{x\}\times\R)\cap\mathcal O_i\ne\emptyset$ because of  $x\in\pi[\mathcal O_i]$. Moreover,  $\{x\}\times\R\subseteq U$  due to
$X\times\R\subseteq U$. Therefore, $U_1$ and $U_2$ partition $\{x\}\times\R$ and are non-empty and open in  $\{x\}\times\R$. This is a contradiction since  $\{x\}\times\R$ is connected.
\end{proof}
\begin{ex}\label{bb}
{{\bf (Interpolation by Brownian motion)}. \rm   The construction of an interpolating continuous function in the proof of Theorem \ref{cont}, see in particular Figure \ref{c2} (a), resembles Paul L\'{e}vy's
construction of Brownian motion, see e.g.\ \cite[Chapter 1.1.2]{MP10}. 

We shall show that one can indeed choose
$\X\subseteq (0,1]$ and, in a non-trivial way,  $(\Y_x)_{x\in\X}$ 
and then construct by L\'{e}vy's method a Brownian motion $(B_x)_{x\in[0,1]}$ on $(\Om,\F,P)$ such that a.s.\ $(B_x)_{x\in\R}\in\I$. (Here $B_x:=0$ for $x\notin[0,1]$.)
 Like in L\'{e}vy's construction we let $\X$ be the set of dyadic numbers in $(0,1]$, namely the disjoint union
$\X:=\X_{-1}\cup\bigcup_{n\ge 1}\mathcal X_n$, where $X_{-1}:=\{1\}$ and $\X_n:=\big\{k2^{-n}\ \big|\ 1\le k< 2^n,\ \mbox{$k$ is odd}\big\}$ for $n\ge 1$.
For $n=-1,1,2,3,\ldots$ and $x\in\X_n$ let $\Y_x\subset \R$ be a two-sided stationary (in the sense of (STAT), see e.g.\ \cite[Theorem 9.9.1]{KT75}) renewal processes with i.i.d.\ interarrival times whose cumulative distribution function is given by
$F_n(t):= 1-\exp\left(-2^{n-2}t^2\right)$ for $t\ge 0$ and whose probability density function we denote by $f_n$.
Furthermore, we assume (IND). For $x\in\X$ and $z\in\R$ we denote by $Y_x(z)$ the a.s.\ unique element of $\Y_x$ with minimal distance to $z$,  i.e.\ distance $D_x(z)$. 
We then recursively define $B_x$ for $x\in\{0\}\cup\X$ as follows: We set $B_0:=0$ and $B_1:=Y_1(0)$. Having defined
$B_x$ for $x\in \X_m, m<n,$ we set
\[B_x:=Y_x\left(\frac{B_{x-2^{-n}}+B_{x+2^{-n}}}{2}\right)\]
for $x\in\X_n$.
By Lemma \ref{2} and \cite[Theorem 9.9.1]{KT75} for all $x\in\X_n$ and $z\in\R$,
\begin{equation}\label{hh}P[D_x(z)>t]=P[D_x^+(0)>2t]=\frac{\int_{2t}^\infty\int_0^\infty f_n(x^++x^-)\, dx^-\, dx^+}{\int_{0}^\infty\int_0^\infty f_n(x^++x^-)\, dx^-\, dx^+}.
\end{equation}
Since 
$\int_0^\infty f_n(x^++x^-)\, dx^-=1-F_n(x^+)$ and 
\[\int_{2t}^\infty1-F_n(x^+)\ dx^+\ =\ \int_{2t}^\infty e^{-2^{n-2}(x^+)^2}\ dx^+\ =\ 2^{(1-n)/2}\int_{2^{(n+1)/2}t}^\infty e^{-s^2/2}\ ds\]
the quantity in (\ref{hh}) is equal to $P\left[|Z|\ge 2^{(n+1)/2}t\right],$
where $Z$ is a standard normal random variable.
Consequently, by symmetry $Y_x(z)-z$ is normally distributed with mean 0 and variance $2^{-n-1}$ as it should for L\'{e}vy's construction. Continuing as in the proof of \cite[Theorem 1.3]{MP10} the function $(B_x)_{x\in\X}$ can be a.s.\ extended to a standard Brownian motion $(B_x)_{x\in[0,1]}$.} 
\end{ex}
\begin{theorem}\label{mono}{\rm \bf (Increasing, bounded functions)}
Assume (IND) and (1STAT). Then  $P[M]=1$ if 
$\sum_{x\in\X}D_x^+(0)<\infty$ a.s.\  and  $P[M]=0$ otherwise.
\end{theorem}
\begin{proof}
Let $(x_n)_{n\ge 0}$ enumerate $\X$.  
For $n\ge 0$ we denote by $\ph_n$ the permutation of $\{0,\ldots,n\}$ for which 
$\left(x_{\ph_n(i)}\right)_{0\le i\le n}$ is strictly increasing.

To prove the first assertion assume that $(D^+_x(0))_{x\in \X}$ is a.s.\ summable. 
For $y\in\R$ and $n\ge 0$ we denote  by $u_n(y):=y+D^+_{x_n}(y)$ the smallest element of $\Y_{x_n}$ which is $\ge y$. 
For $n\ge 0$ we define the function $f_n:\R\to[0,\infty)$ 
by
$f_n(x):=0$ if $x< x_{\ph_n(0)}$ and
\[f_n(x):=u_{\ph_n(i)}\left(u_{\ph_n(i-1)}\left(\ldots \left(u_{\ph_n(1)}\left(u_{\ph_n(0)}(0)\right)\right)\ldots\right)\right)\]
if $0\le i<n$ and $x_{\ph_n(i)}\le x< x_{\ph_n(i+1)}$ or if $i=n$ and $x_{\ph_n(n)}\le x$,  see Figure \ref{monof}. 
\begin{figure}[t]
 \psfrag{a}{$x_{\ph_4(0)}$}
 \psfrag{b}{$x_{\ph_4(1)}$}
\psfrag{c}{$x_{\ph_4(2)}$}
\psfrag{d}{$x_{\ph_4(3)}$}
\psfrag{e}{$x_{\ph_4(4)}$}
\psfrag{f}{$f_4$}
\epsfig{figure=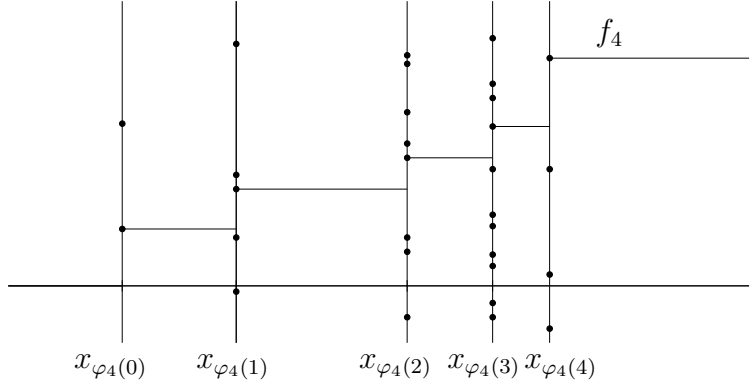,height=140pt}\hspace*{5mm}
\caption{\footnotesize $f_n$ is the smallest increasing function $\R\to[0,\infty)$ with $f_n(x_i)\in\Y_{x_i}$ for all $0\le i\le n$. Here $n=4$ and (PPP) hold.}\label{monof}
\end{figure}
Then $f:=\sup_{n}f_n$
satisfies the requirements formulated in the definition of $M$, as we shall explain now.

Firstly, $f$ is increasing since each $f_n$ is increasing. 
Secondly, $f$ is a.s.\ bounded. Indeed,
due to (IND), (1STAT)  and Lemma \ref{2}, $\|f_n\|_\infty=f_n(x_{\ph_n(n)})$ has the same distribution as $D_{x_0}^+(0)+\ldots+D_{x_n}^+(0)$. Moreover, $(f_n(x))_{n\ge 0}$ is for each $x\in\R$  an increasing sequence.
Consequently, $\|f_n\|_\infty$ increases for $n\to\infty$ towards $\|f\|_\infty$. Hence $\|f\|_\infty$ has the same distribution as $\sum_{x\in\X}D^+_x(0)$ and is therefore  a.s.\ finite by assumption.  
Thirdly, we have for all $i\ge 0$ that $f(x_i)\in \mathcal Y_{x_i}$ a.s.\ since $f_n(x_i)\in \mathcal Y_{x_i}$ for all $n\ge i$ and $\mathcal Y_{x_i}$ is  closed. 

For the proof of the second statement of the theorem assume that $P[\sum_{x} D^+_x(0)=\infty]>0$.
For all $a,b\in\Z$ with $a<b$ define
\[ M_{a,b}:=\left\{
\exists f\in\I:\ f\mbox{ is increasing},\ f[\R]\subseteq [a,b]\right\}.
\]
Note that $M$ is the union of all $M_{a,b}$, $a<b$. Therefore, it suffices to show that $P[M_{a,b}]=0$ for all $a<b$.
For all $n\ge 0$,
\begin{eqnarray*}
M_{a,b}&\subseteq&\left\{\exists f:\R\to[a,b]:\ f\mbox{ is increasing, }\forall\ 0\le i\le n\ f(x_i)\in \mathcal Y_{x_i}
\right\}.
\end{eqnarray*}
Any increasing $f:\R\to[a,b]$ with $f(x_i)\in \mathcal Y_{x_i}$ for all $0\le i\le n$ 
must satisfy
\[
f(x_{\ph_n(i)})\ge S^{(n)}_i:=\inf\left\{y\ge S_{i-1}^{(n)}\ \big|\ y\in \mathcal Y_{\ph_n(i)}\right\} \qquad(0\le i\le n),
\]
where $S_{-1}^{(n)}:=a$.
Therefore, $M_{a,b}\subseteq\{S^{(n)}_n\le b\}\in\F$ for all $n\in\N$. Hence,
by (IND), (1STAT) and Lemma \ref{2}, 
\[P[M_{a,b}]\le \inf_{n\in\N}P\left[S^{(n)}_n\le b\right]=\inf_{n\in\N}P\left[\sum_{i=0}^nD^+_{x_i}(0)\le b-a\right]\le
P\left[\sum_{x\in\X}D^+_{x}(0)<\infty\right],\]
which is, by assumption, strictly less than 1 and therefore, by Kolmogorov's zero-one law, equal to 0.
\end{proof}
\begin{rem}\label{+-}{\rm In Theorem \ref{mono} one can replace $D_x^+(0)$ by $D_x(0)$ due to 
Lemma \ref{2}.}
\end{rem}
\begin{ex}\label{3}{\rm Assume (PPP). Then, by the three series theorem and Theorem \ref{mono},
 $P[M]=1$ if 
\begin{equation}\label{co}
\sum_{x\in\X}\frac{1}{\la_x}<\infty
\end{equation} and
$P[M]=0$ else. In particular, (\ref{co}) implies $P[BV]=1$ since $M\subseteq BV$. For the reverse implication we have the following partial result.}
\end{ex}
\begin{theorem}{\bf (Functions of bounded variation)}\label{bv} Assume (PPP) and $\X=\N$. Then   $P[BV]=0$ if any one of the following conditions {\rm (\ref{a})--(\ref{c})} holds:
\begin{eqnarray}
\label{a}&& (\la_n)_{n\in\N}\quad\mbox{is increasing and}\quad \sum_{n\in\N}\frac{1}{\la_n}=\infty.\\
\label{b}&& \liminf_{n\to\infty}\la_n<\infty.\\
\label{c}&& \lim_{n\to\infty}\la_n=\infty\quad\mbox{and}\quad
\limsup_{n\to\infty}\frac{n}{\la_{(n)}}>0,\quad\mbox{where $(\la_{(n)})_{n\in\N}$ is an  }\\
\nonumber&&\mbox{increasingly ordered permutation of $(\la_n)_{n\in\N}$.}
\end{eqnarray}
\end{theorem}
\begin{problem}\label{nlogn}{\rm
(a) Theorem \ref{bv} and Example \ref{3} do not determine $P[BV]$ in some  cases in which $(\la_n)_{n\ge 0}$ is a permutation of, say, $(n\log n)_{n\ge 2}$.
(b) Can Theorem \ref{bv} be extended to $\X$ whose closure contains an interval?} 
\end{problem}
The following non-probabilistic lemma is needed for case (\ref{c}).
\begin{lemma}\label{impl}
Let $(\mu_n)_{n\ge 0}$ be a sequence of positive  numbers which monotonically decreases to 0.
Then 
\begin{equation}\label{posi}
\limsup_{n\to\infty}\ n\mu_{n}>0
\end{equation}
implies
\begin{equation}\label{mini}
\forall \ph:\N_0\to\N_0\ \mbox{bijective}\quad \exists M\subseteq \N_0\quad  \sum_{n\in M}
 \min_{m\in M:\ m\le n}\mu_{\ph(m)}=\infty.
\end{equation}
\end{lemma}
\begin{rem}{\rm 
In fact, (\ref{posi}) and (\ref{mini}) are equivalent.}
\end{rem}
\begin{proof}[Proof of Lemma \ref{impl}] 
Let $\eps:=\limsup_{n\to\infty}\ n\mu_{n}>0$ and let $\ph$ be a permutation of $\N_0$.
We inductively construct a sequence $(M_k)_{k\ge 0}$ of  finite  sets
$M_k\subset \N_0$ such that for all $0\le i<j$,
\begin{eqnarray}\label{van}
\forall m\in M_i\ \forall n\in M_j&:&  m<n,\\
&& \ph(m)<\ph(n)\quad \mbox{and}\label{van2}
\end{eqnarray}
\begin{equation}
\#M_j\min_{m\in M_j} \mu_{\ph(m)}\ge \eps/2.\label{van3}
\end{equation}
The induction starts with  $M_0:=\{0\}$. Let $k\ge 1$ and assume that we have already defined  finite  sets $M_0,\ldots,M_{k-1}$ which fulfill conditions (\ref{van})--(\ref{van3}) for all $0\le i<j<k$. (These conditions are void for $k=1$.)
We set $m_k:=\max\{\ph(a)\mid 0\le a\le\max M_{k-1}\}$ and choose $n_k>m_k$ large enough such that
$(n_k-m_k)\mu_{n_k}>\eps/2$. Then $M_k:=\{\ph^{-1}(a)\mid m_k<a\le n_k\}$ is finite. Moreover, (\ref{van})--(\ref{van3}) hold for all $0\le i<j\le k$ as well.
Indeed, by induction hypothesis we only need to consider the case $j=k$. Property (\ref{van3}) follows from the definition of $M_k$ and the fact that $\mu_a$ decreases in $a$. For the proof of (\ref{van}) and (\ref{van2}) let  $m\in M_i$ and $n\in M_k$. The definition of $M_k$ implies $m_k<\ph(n)$. Therefore, by 
definition of $m_k$,  $\ph(a)<\ph(n)$ for all $0\le a\le \max M_{k-1}$.
In particular,  $a\ne n$  for all $0\le a\le \max M_{k-1}$ and therefore $a<n$ for all $0\le a\le \max M_{k-1}$.
 Due to (\ref{van}) this applies to $a=m$ and yields $\ph(m)<\ph(n)$ and $m< n$.

Having constructed the sequence $(M_k)_{k\ge 0}$
we set $M:=\bigcup_{k\ge 1}M_k$. By disjointness, see (\ref{van}), 
\begin{eqnarray*}
\lefteqn{
\sum_{n\in M}
\min_{m\in M:\ m\le n} \mu_{\ph(m)}
= \sum_{j\ge 1}\sum_{n\in M_j}
\min_{m\in M:\ m\le n} \mu_{\ph(m)}}\\
&\stackrel{(\ref{van})}{=}&\sum_{j\ge 1}\sum_{n\in M_j}
\min_{m\in M_1\cup\ldots \cup M_j:\ m\le n} \mu_{\ph(m)}
\ \stackrel{(\ref{van2})}{=}\ \sum_{j\ge 1}\sum_{n\in M_j}
\min_{m\in M_j:\ m\le n} \mu_{\ph(m)}\\
&\ge& \sum_{j\ge 1}\sum_{n\in M_j}
\min_{m\in M_j} \mu_{\ph(m)}
\ \stackrel{(\ref{van3})}{\ge}\ \sum_{j\ge 1} \eps/2\ =\ \infty.
\end{eqnarray*}
\end{proof}

For the proof of Theorem \ref{bv} we first show in Lemma \ref{BV2} that whenever there is  a function of bounded variation in $\I$ then there is also another such  function which is only ``jumping between nearest neighbors".
More precisely, we recursively define the random function
$g:\bigcup_{n\in\N_0}\{+,-\}^n\to\R$ such that $g(s_1,\ldots,s_n)\in\Y_n$ for $n\ge 1$
by setting $g(\la):=0$, where $\la$ is the empty sequence, i.e.\ the only element of $\{+,-\}^0$, and 
\begin{equation}\label{alb}
g(st):=g(s)+\left(tD_{|st|}^t(g(s))\right)\qquad\mbox{for}\ s\in \bigcup_{n\in\N_0}\{+,-\}^n, t\in\{+,-\},
\end{equation}
where $st$ is the concatenation of $s$ and $t$ and $|\cdot|$ denotes the length of a sequence. For any (finite or infinite) sequence $s=(s_i)_{1\le i<N+1}$ of length $N\le \infty$ we let
\[g^{s}:=(g(s_1,\ldots,s_i))_{0\le i<N+1}.\]
Moreover, for any finite of infinite sequence $h=(h_0,h_1,\ldots)$ of real numbers we define the total variation of $h$  
as $V(h):=\sum_{i}|h_{i}-h_{i-1}|.$
\begin{lemma}\label{BV2} Let $\X=\N$. Then 
$BV\subseteq BV_2:=\left\{\exists s\in\{+,-\}^{\N}: V(g^s)<\infty\right\}.$
\end{lemma}
\begin{proof} 
Let $f\in\I$ be of bounded variation. Set $f_0:=0$ and $f_i:=f(i)$ for $i\in\N$.
We define inductively for $i\in \N$,
\begin{equation}\label{ss}
s_i:={\rm sign}\left(f_i-g(s_1,\ldots,s_{i-1})\right)\in\{+,-\},\quad\mbox{where ${\rm sign\, 0}:=+$,}
\end{equation}
and set $s:=(s_i)_{i\in\N}$. Thus $(g_i)_{i\ge 0}:=g^s$ is the sequence which starts at 0 and tries to trace $(f_i)_{i\ge 0}$ but is restricted to making only the smallest possible jumps up or down.
To prove $V(g^s)<\infty$ we consider for all $n\in\N$ the telescopic sum
\begin{eqnarray}\nonumber
\lefteqn{V(f_0,\ldots,f_n)-V(g_0,\ldots,g_n)\ =\ V(g_0,f_1,\ldots,f_n)-V(g_0,\ldots,g_n)}\\
&=&\sum_{i=1}^nV(g_0,\ldots,g_{i-1},f_i,\ldots,f_n)-V(g_0,\ldots,g_{i},f_{i+1},\ldots,f_n).\label{in}
\end{eqnarray}
\begin{figure}[t]
 \psfrag{a}{$i-1$}
 \psfrag{b}{$i$}
\psfrag{c}{$i+1$}
\psfrag{f}{$f_i$}
\psfrag{g}{$g_i$}
\epsfig{figure=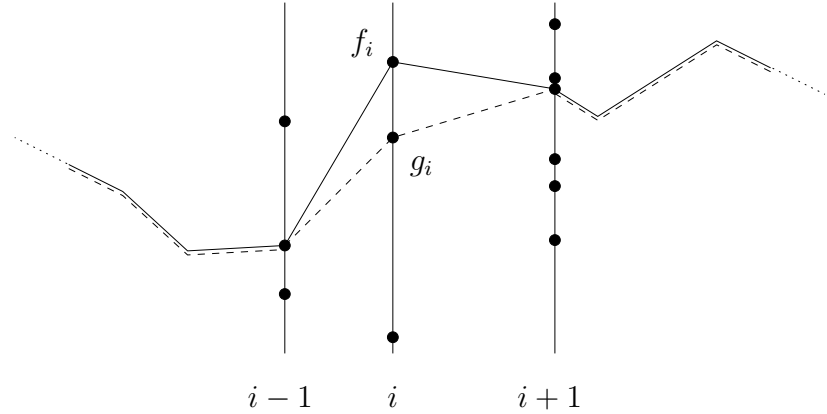,height=155pt}\hspace*{5mm}
\caption{\footnotesize The dashed graph has a smaller total variation than the solid graph.}\label{des}
\end{figure}
We shall show that all the summands in (\ref{in}) are non-negative, see Figure \ref{des}. Since the  sequences $(g_0,\ldots,g_{i-1},f_i,\ldots,f_n)$ and $(g_0,\ldots,g_{i},f_{i+1},\ldots,f_n)$  differ only in the $i$-th term the $i$-th summand in (\ref{in}) is equal to
\begin{eqnarray}
\lefteqn{|f_i-g_{i-1}|+|f_{i+1}-f_i|-|g_i-g_{i-1}|-|f_{i+1}-g_i|}\nonumber\\
&\stackrel{(\ref{ss}),(\ref{alb})}{=}&s_i(f_i-g_{i-1})+|f_{i+1}-f_i|-D_i^{s_i}(g_{i-1})-|f_{i+1}-g_{i}|\nonumber\\
&=&s_i\left(f_i-(g_{i-1}+(s_iD_i^{s_i}(g_{i-1})))\right)+|f_{i+1}-f_i|-|f_{i+1}-g_{i}|\nonumber\\
&\stackrel{(\ref{alb})}{=}&s_i\left(f_i-g_i\right)+|f_{i+1}-f_i|-|f_{i+1}-g_{i}|\nonumber\\
&=&\left|f_i-g_{i}\right|+|f_{i+1}-f_i|-|f_{i+1}-g_i|,\label{last}
\end{eqnarray}
where we used in (\ref{last}) that by definition $f_i,g_i\in\Y_i$ with $f_i\ge g_i\ge g_{i-1}$ if $s_i=+$ and $f_i\le g_i\le g_{i-1}$ if $s_i=-$.
 Due to the triangle inequality the expression in
(\ref{last}) is non-negative. Hence, by (\ref{in}),  $V(g_0,\ldots,g_n)\le V(f_0,\ldots,f_n)$
and therefore
\[V(g^s)=\sup_{n\in\N}V(g_0,\ldots,g_n)\le \sup_{n\in\N}V(f_0,\ldots,f_n)=V\left((f_n)_{n\ge 0}\right),\]
which is less than or equal to the total variation of $f$, which is finite.
\end{proof}

\begin{proof}[Proof of Theorem \ref{bv}.]
First consider  case (\ref{a}). By Lemma \ref{BV2} is suffices to show $P[BV_2]=0$. 
The proof goes along  the same lines as part of the proof of \cite[Theorem 3]{PemPer94}, which gives a criterion for explosion of first-passage percolation on spherically symmetric trees.
By definition of $BV_2$,
\begin{eqnarray}\nonumber
P[BV_2]&=&P\left[\inf_{s\in\{+,-\}^\N}V(g^s)<\infty\right]=
P\left[\inf_{s\in\{+,-\}^\N}\sup_{n\in \N}V(g^{(s_1,\ldots,s_n)})<\infty\right]\\
&\le&P\left[\sup_{n\in \N}\inf_{s\in\{+,-\}^\N}V(g^{(s_1,\ldots,s_n)})<\infty\right]=P\left[\lim_{n\to\infty} \min_{s\in\{+,-\}^n}V(g^s)<\infty\right].\label{epa}
\end{eqnarray}
For all $n\in \N$ and $s\in\{+,-\}^n$,
\begin{equation}\label{na}
V(g^s)=\sum_{i=1}^n|g(s_1,\ldots,s_{i})-g(s_1,\ldots,s_{i-1})|\stackrel{(\ref{alb})}{=}\sum_{i=1}^n D_i^{s_i}(g(s_1,\ldots,s_{i-1})).
\end{equation}
This is the sum of independent exponentially distributed random variables with respective parameter $\la_i$, see Example \ref{paul}.
To apply standard large deviation estimates we weight the summands to make them i.i.d.\ and consider first the sum
\begin{equation}\label{arg}
V^*(g^s):=\sum_{i=1}^n \la_i D_i^{s_i}(g(s_1,\ldots,s_{i-1}))
\end{equation}
of $n$ independent random variables which are all exponentially distributed with parameter 1. We choose $\eps>0$ such that $2 e^{2\eps}<3$, denote by $+^n$ the only element of $\{+\}^n$ and get  
\[P\left[\min_{s\in\{+,-\}^n}V^*(g^s)<\eps n\right]\le 2^nP\left[V^*(+^n)<\eps n\right]\le
2^ne^{2\eps n}E\left[e^{-2V^*(+^n)}\right]=\left(\frac{2e^{2\eps}}{3}\right)^n,\]
which is summable in $n$. Therefore, by the Borel Cantelli lemma, there is  a.s.\ some random $N$ such that 
\begin{equation}\label{une}
M_n^*:=\min_{s\in\{+,-\}^n}V^*(g^s)\ge \eps n\qquad\mbox{for all $n\ge N$.}
\end{equation}
Recalling (\ref{na}) we have for all $n\in \N$ and $s\in\{+,-\}^n$,
\begin{eqnarray}
V(g^s)&=&\sum_{i=1}^n\la_i^{-1}\left(\sum_{j=1}^i\la_j D_j^{s_j}(g(s_1,\ldots,s_{j-1}))-\sum_{j=1}^{i-1}\la_j D_j^{s_j}(g(s_1,\ldots,s_{j-1}))\right)\nonumber \\
&\stackrel{(\ref{arg})}{=}&\sum_{i=1}^n\la_i^{-1}\left(V^*(g^{(s_1,\ldots,s_i)})-V^*(g^{(s_1,\ldots,s_{i-1})})\right) \nonumber\\
&=&\sum_{i=1}^n\la_i^{-1}V^*(g^{(s_1,\ldots,s_i)})-\sum_{i=1}^{n-1}\la_{i+1}^{-1}V^*(g^{(s_1,\ldots,s_i)})\nonumber \\
&=& \la_n^{-1}V^*(g^s)+\sum_{i=1}^{n-1}(\la_i^{-1}-\la_{i+1}^{-1})V^*(g^{(s_1,\ldots,s_i)}).\label{rn}
\end{eqnarray}
Due to assumption (\ref{a}), $\la_i^{-1}-\la_{i+1}^{-1}\ge 0$. Therefore, the right hand side in (\ref{rn}) is
\begin{eqnarray}\nonumber
&\ge&\la_n^{-1}V^*(g^s)+\sum_{i=1}^{n-1}(\la_i^{-1}-\la_{i+1}^{-1})M^*_i\\ \nonumber
&=&\la_n^{-1}V^*(g^s)+\sum_{i=1}^{n-1}(\la_i^{-1}-\la_{i+1}^{-1})\eps i\\ \nonumber
&&+\ \sum_{i=1}^{(N\wedge n)-1}(\la_i^{-1}-\la_{i+1}^{-1})(M^*_i-\eps i)+ \sum_{i=N\wedge n}^{n-1}(\la_i^{-1}-\la_{i+1}^{-1})(M^*_i-\eps i)\\ \label{no}
&\stackrel{(\ref{une})}{\ge}&\la_n^{-1}V^*(g^s)+\eps\sum_{i=1}^{n-1}(\la_i^{-1}-\la_{i+1}^{-1}) i-c_1(N)+0,\qquad\mbox{where}\\
c_1(N)&:=&\eps \sum_{i=1}^{N-1}(\la_i^{-1}-\la_{i+1}^{-1})i.\nonumber
\end{eqnarray}
The second summand in (\ref{no}) is equal to
\[\eps\left(\sum_{i=1}^{n-1}\frac{i}{\la_i}-\sum_{i=2}^{n}\frac{i-1}{\la_i}\right)=\eps\left(\frac{1}{\la_1}-\frac{n-1}{\la_n}+\sum_{i=2}^{n-1}\frac{1}{\la_i}\right)=\eps\left(-\frac{n-1}{\la_n}+\sum_{i=1}^{n-1}\frac{1}{\la_i}\right).\]
Therefore, the right hand side of (\ref{no}) is equal to
\begin{equation}\label{su}
\eps\sum_{i=1}^{n-1}\frac{1}{\la_i}-c_1(N)+\frac{V^*(g^s)-\eps(n-1)}{\la_n}\ge
\eps\sum_{i=1}^{n-1}\frac{1}{\la_i}-c_1(N)-c_2(N)\\
\end{equation}
since
\[ \frac{V^*(g^s)-\eps(n-1)}{\la_n}\ge \frac{M^*_n-\eps n}{\la_n}\stackrel{(\ref{une})}{\ge} \frac{-\eps N}{\la_1}=:c_2(N).\]
The right hand side of (\ref{su}) does not depend on the particular choice of   $s\in\{+,-\}^n$ and tends a.s.\ to $\infty$ as
$n\to\infty$ due to (\ref{a}). Consequently, the right most side of (\ref{epa}) is 0, which completes the proof in case (\ref{a}).

In case (\ref{b}) there is some finite constant $c$ such that $\X':=\{x\in\X:\ \la_x<c\}$ is infinite. Since $\X'\subseteq\X$ and $(\la_x)_{x\in\X'}\le (c)_{x\in\X'}$ the claim follows from (\ref{a}) and monotonicity (Remark \ref{mm}).

Case (\ref{c}) is treated similarly. We choose a permutation $\ph$ of $\N$ such that $\la_{(n)}=\la_{\ph^{-1}(n)}$ for all $n$.
Applying Lemma \ref{impl} to $\mu_n:=1/\la_{(n)}$ gives a set $\X':=M\subseteq \N_0$ such that  the increasing sequence $(\la'_x)_{x\in\X'}$ with $\la'_x:=\max_{n\in\X':n\le x}\la_n\ge\la_x$ satisfies  $\sum_{x\in\X'}1/\la'_x=\infty$. 
The statement now follows from case (\ref{a}) by monotonicity (Remark \ref{mm}).
\end{proof}
\begin{theorem} {\rm \bf (Lipschitz functions)} \label{li}
Denote by $p_c$ the threshold for oriented site percolation on the square lattice $\Z^2$.
Then there is 
\begin{equation}\label{lr}
\la_c\in-\frac{[1,3]}{2}\ln (1-p_c) ,
\end{equation}
with the following property:
If (PPP) holds, $\X=\Z, K>0$ and $\la>0$ is such that $\la_x=\la$ for all $x\in\X$ then
\begin{equation}\label{crit}
P[L_K]=0\quad\mbox{if $\la<\la_c/K$ and}\quad P[L_K]=1\quad\mbox{if $\la>\la_c/K$}.
\end{equation}
\end{theorem}
\begin{proof} By independence of $(\Y_x)_{x\in\X}$ and ergodicity of each $\Y_x$, $P[L_K]\in\{0,1\}$, see Remark \ref{0E}. Since $P[L_K]$ is increasing in $\la$ for every $K$, see Remark \ref{mm}, there exists for all $K$ some $\la_c(K)\in[0,\infty]$ with property (\ref{crit}). Scaling $V$ by $(x,y)\mapsto
(x,y/K)$ reduces the problem to considering Poisson point processes of intensity $\la K$ and Lipschitz functions with Lipschitz constant 1. This implies that $\la_c(K)$ in fact does not depend on $K$.

For the proof of (\ref{lr}) we partition $\Z\times\R$ into 
the line segments $S(x,i):=\{x\}\times [4i+(-1)^x,4(i+1)+(-1)^x)$ of length 4,  where $x,i\in\Z$. 
Consider the graph $G$ with vertices $S(x,i)$, and oriented edges 
$(S(x,i),S(x+1,i))$ and $(S(x,i),S(x+1,i+(-1)^x))$, where $x,i\in\Z$.
This graph is isomorphic to the oriented square lattice, see Figure \ref{op} (a). 
\begin{figure}[t]
\epsfig{figure=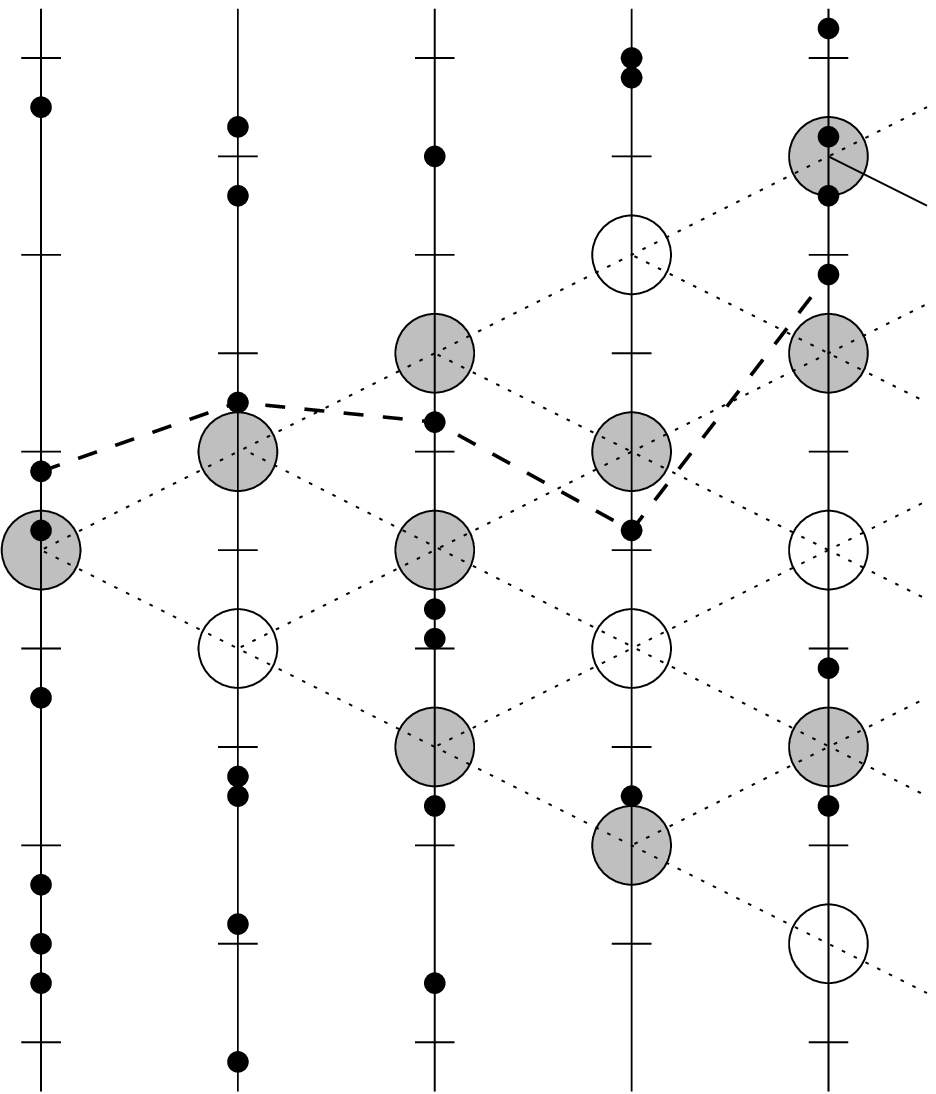,height=210pt}
\hspace*{5mm}\epsfig{figure=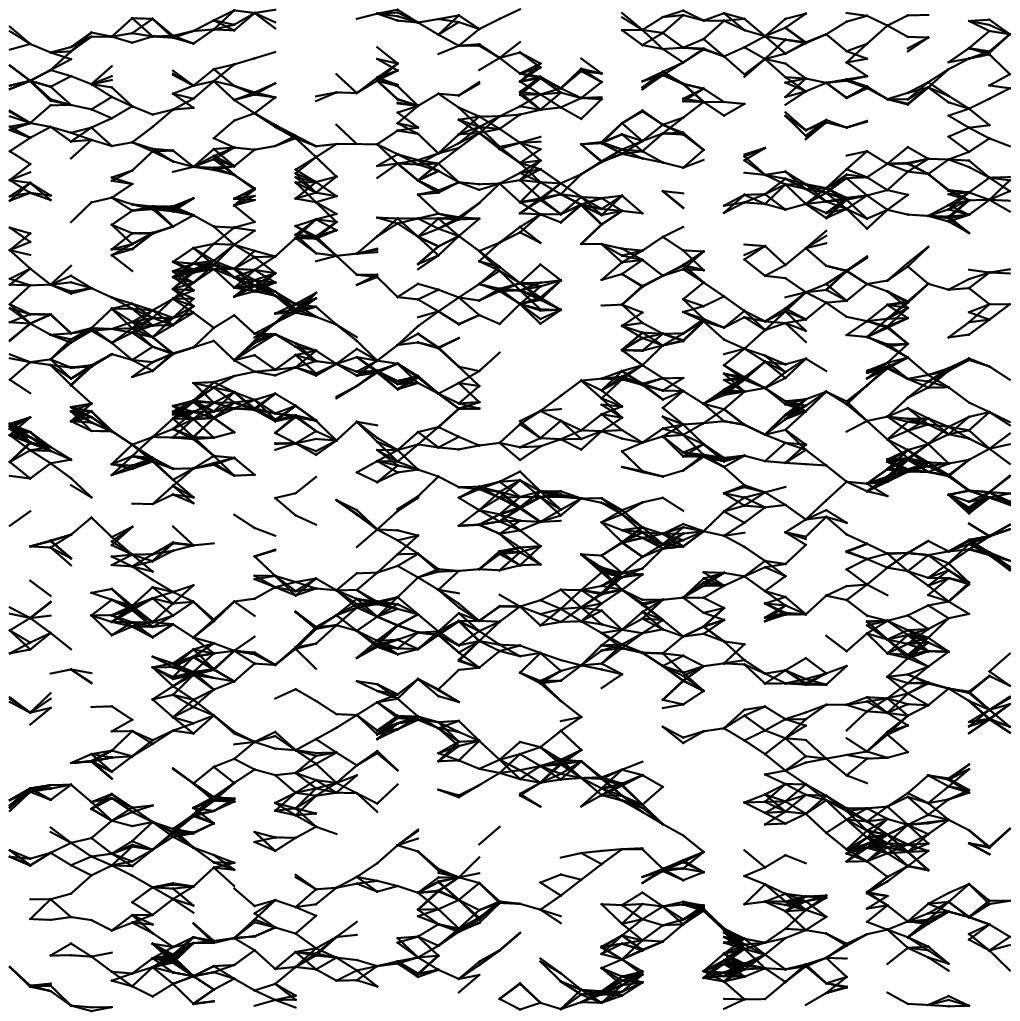,height=215pt}
\caption{\footnotesize Here we assume (PPP) and $\X=\Z$. (a) Vertices of a quadrant of $G$ are indicated as discs, the edges connecting them are dotted. Discs are shaded, i.e.\ open, if and only if there is a point of the Poisson point process in the corresponding interval. The graph of some Lipschitz function accompanying an open directed path is dashed. (b)  Here $\la_x=1$ for all $x\in\X$. The figure shows those straight line segments whose end points $(x,y_x)$ and $(x+1,y_{x+1})$ satisfy $x\in\{1,2,\ldots,50\};\ y_x\in\Y_x\cap[0,50],\ y_{x+1}\in\Y_{x+1}\cap[0,50]$ and $|y_{x+1}-y_x|\le K=1$. Do these line segments contain the graph of a function defined on [1,50]? }\label{op}
\end{figure}
We declare the vertex $S(x,i)$
to be open  if it intersects $\{x\}\times\Y_x$ and to be closed otherwise.
Denote by $\mathcal O$ the event that there is a double infinite  directed path in $G$ between open nearest neighbors.
Any such path induces some $f\in\I$ with Lipschitz constant $6$. Therefore, $\mathcal O\subseteq L_{6}$.
Since any vertex is open with probability $1-e^{-4\la}$ this yields the upper bound on $\la_c=\la_c(6)$ in (\ref{lr}). 
Conversely, any $f\in\I$ with Lipschitz constant $2$ induces an open  double infinite directed path in $G$. Hence $L_{2}\subseteq \mathcal O$. This implies the lower bound on $\la_c=\la_c(2)$ in (\ref{lr}).
\end{proof}
\begin{rem}{\rm Substituting Bishir's \cite{B63} \cite[Theorem 4.1]{PeaFl05} lower bound $p_c\ge 2/3$    and 
Liggett's  \cite{L95} upper bound $p_c\le 3/4$ 
into (\ref{lr}) yields $0.549\le \la_c\le 2.08$. The lower bound can be improved by directly applying the methods described in \cite[Section 6]{Du84}.}
\end{rem}
\begin{problem}{\rm What is the exact value of $\la_c$? Critical thresholds in (oriented) percolation are rarely explicitly known. However, due to
Monte-Carlo simulations,   see also Figure \ref{op} (b), we conjecture $\la_c=1$.}
\end{problem}
In the following simple example the corresponding $\la_c$ can be computed explicitly.
\begin{ex}{\rm Let $\X=\Z$ and $\la,K>0$.  Let $(U_n)_{n\in\Z}$ be independent and uniformly distributed on $[0,1]$ and set
$\Y_n:=(U_n+\Z_n)/\la$ for $n\in\Z$. (Note the similarity of $V$ to perforated toilet paper.) Then for any $n\in\Z$ and $z\in\R$ there is some $y\in\Y_n$ with $|z-y|\le K$ if  $2K\ge 1/\la$. If  $2K< 1/\la$ then there is at most one such $y$ and with positive probability no such $y$. Hence $P[L_K]=1$ if $\la\ge 1/(2K)$ and $P[L_K]=0$ else. Thus in this case (\ref{crit}) holds for $\la_c=1/2$.}
\end{ex}
Next we consider some of the smoothest functions, polynomials and real analytic functions.
\begin{prop}{\rm \bf (Polynomials)} \label{po} Assume (IND) and (STAT) and that all $\Y_x$, $x\in\X$, are a.s.\ countable.
Then $P[P_m]=0$ for all $m\in\N_0$.
\end{prop}
\begin{proof} Let $m\in\N_0$ and let $x_0,\ldots,x_{m+1}\in\X$ be pairwise distinct.
For any $y=(y_0,\ldots,y_{m})\in\prod_{i=0}^{m}\Y_{x_i}$ there is exactly one polynomial $f_y$ of degree $m$ with
$f_y(x_i)=y_i$ for $i=0,\ldots,m$. For all $K\in\N$ the set $\prod_{i=0}^{m}\left(\Y_{x_i}\cap[-K,K]\right)$ is countable and compact. Hence, 
$G_K:=\{f_y(x_{m+1}):\ y\in\prod_{i=0}^{m}\left(\Y_{x_i}\cap[-K,K]\right)\}$ is a.s.\ countable and, by continuity, compact as well.
If $P_m$ occurs then $G_K\cap\Y_{x_{m+1}}\ne\emptyset$ for some $K\in\N$. By (IND) and (STAT), the closed set
$G_K\cap\Y_{x_{m+1}}$ has for all $a\in\R$ the same distribution as $G_K\cap(\Y_{x_{m+1}}+a)$. Therefore,
\begin{eqnarray*}
P[P_m]&\le& \sum_{K\in\N}P[G_K\cap\Y_{x_{m+1}}\ne \emptyset]
\ =\ \sum_{K\in\N}\int_0^1 E\left[\won_{\{G_K\cap(\Y_{x_{m+1}}+a)\ne \emptyset\}}\right]\ da.
\end{eqnarray*}
Using Fubini's theorem and denoting by $\la$ the Lebesgue measure on $\R$, we get
$P[P_m]\le\sum_K
E\left[\la\left(G_K-\Y_{x_{m+1}}\right)\right]\ =\ 0$
since $G_K-\Y_{x_{m+1}}:=\{a-b\mid a\in G_K, b\in\Y_{x_{m+1}}\}$ is a.s.\ countable.
\end{proof}
Of course, the assumption of countability in Proposition \ref{po} cannot be dropped. Here is a nontrivial example.
Let $(x_n)_{n\in\N}$ enumerate $\X$ and let $1>\ell_1\ge\ell_2\ge\ldots\ge \ell_n\to 0$ as $n\to\infty$. We consider two different families $(\Y_x)_{x\in\X}$. The first one is  of the type described in (\ref{ano}) and is given by 
\begin{equation}\label{pcb}
\Y_{x_n}=[\ell_n,1]+\Z+U_n,\quad \mbox{where $(U_n)_{n\ge 1}$ is i.i.d., $U_n\sim{\rm Unif}[0,1],$}
\end{equation}
cf.\ Figure \ref{1} (b). The second one consists of complements of Boolean models:
\begin{equation}\label{pcc}
\begin{array}{l}
\Y_{x_n}=(]0,\ell_n[+\Y'_n)^c,\quad \mbox{where $(\Y'_n)_{n\ge 1}$ are independent}\\ 
\mbox{Poisson point processes with intensity 1.}
\end{array}
\end{equation}
In case (\ref{pcb}), $P[0\in\I]=\prod_{n\ge 1}(1-\ell_n)$, while in case 
(\ref{pcc}), $P[0\in\I]=\prod_{n\ge 1}e^{-\ell_n}$. Therefore,
 in either case 
$P[0\in\I]=0$ if and only if $\sum_n\ell_n=\infty$. Obviously, $\{0\in\I\}\subseteq P_0$. However, the following theorem due to L.\ A.\ Shepp shows that these two events might differ by more than a null set. To recognize it recall (\ref{int}) and take complements. 
\begin{theorem}\label{shepp}{\bf (Constant functions, \cite{Sh72a}, \cite[(42)]{Sh72b}).}
Assume (\ref{pcb}) or (\ref{pcc}). Then
$P[P_0]=0$ if and only if
\begin{equation}\label{gad}
\sum_{n\ge 1}\frac{1}{n^2}\exp\left(\sum_{i=1}^n\ell_i\right)=\infty.
\end{equation}
\end{theorem}
\begin{rem}\label{no0}
{\rm Note that in case (\ref{pcb}) there is no zero-one law like in Remark \ref{0E}: $0<P[P_0]<1$ is possible. In fact, $P[P_0]=1$ if and only if  $\sum_n\ell_n\le 1$, which is not the opposite of (\ref{gad}).}
\end{rem}
\begin{problem}\label{speed} {\rm Let $m\in\N$. Assuming (\ref{pcb}) or (\ref{pcc}), find conditions which are necessary and sufficient for $P[P_m]=0$ (respectively, $P[P_m]=1$).

For $m=1$ and case (\ref{pcb}) (see also Figure \ref{1} (b)) this problem can be phrased in terms of random coverings of a circle in the spirit of \cite{Sh72a} and \cite{JS08} as follows:
Arcs of length $\ell_x\ (x\in\X)$ are thrown independently and uniformly on a circle of unit length
and then rotate at respective speed $x$ around the circle. Give a necessary and sufficient condition in terms of $(\ell_x)_{x\in\X}$ and $\X$ under 
which there is a.s.\ no (resp., a.s.\ at least one) point in time at which the circle is not completely covered by the arcs. In other words: Under which conditions 
is there a.s.\ no (resp., a.s.\ at least one) random, but  constant speed at which one can drive along a road with infinitely many independent traffic lights without 
ever running into a red light? 
}
\end{problem}
\begin{prop}{\rm \bf (Real analytic functions)}\label{ana}
If $\X$ is locally finite then $P[A]=1$.
\end{prop}
\begin{proof} Choose $y_x\in\Y_x$ for all $x\in\X$.  
By \cite[Theorem 15.13]{Ru87} there is an entire function $g:\C\to\C, g(z)=\sum_{n\ge 0}a_nz^n,$ such that $g(x)=y_x$ for all $x\in\X$. Its real part $\Re(g(z))=\sum_{n\ge 0}\Re(a_n)z^n$ restricted to $\R$ is real analytic and  takes values $y_x$ at $x\in\X$ as well.
\end{proof}
We conclude by suggesting some further directions of research.
\begin{problem}{\rm Theorem \ref{li} and  Proposition \ref{ana} deal only with locally finite $\X$. What can be said about $P[L_K]$ and $P[A]$ for more general $\X$?
It is easy to see that even in case (PPP) with constant intensities $\la_x=\la$
any general criterion for, say, $P[L_K]=0$  
would need to depend not only on $\la$ but also on $\X$ itself.
If $\X$ is for example bounded 
then for any $K>0$, $P[L_K]\le P[C]=0$ by Theorem  \ref{cont}, no matter how large $\la$ is, in contrast to (\ref{crit}).
}
\end{problem}
\begin{problem}{\rm 
One might consider other types of interpolating functions.
For example, under which conditions are there functions $f\in\I$, which are (a)  
continuous and  monotone at the same time or (b) H\"older continuous or (c) $k$-times continuously differentiable? Extensions to higher dimension might be possible as well.
}
\end{problem}
\begin{problem}\label{rm}{\rm
{\bf (More than just one function)}
Let $\X\subset[0,1]$ and fix $(\la_x)_{x\in\X}$ with $\la_x>0$. 
Under which conditions is there a simple point process $N=\sum_{i\in\N}\delta_{f_i}$ (in the sense of \cite[Definition 7.1.VII]{DV-J88}) on the space $C([0,1])$ of continuous functions on $[0,1]$ (or any other suitable space of regular functions) such that  
(a) for all $x\in\X$ the points $f_i(x),\ i\in\N,$ are pairwise distinct, (b) there are independent homogeneous Poisson point processes $\Y_x,\ x\in\X,$ with intensities $\la_x$ such that $\{f_i(x):\ i\in\N\}\subseteq\Y_x$ for all $x\in\X$ and (c) 
the ``vertically shifted" point process $\sum_{i\in\N}\delta_{f_i+y}$ has for all $y\in\R$ the same distribution as $N$?
}
\end{problem}

\bibliographystyle{amsalpha}
\vspace*{1mm}

\noindent 
Eberhard Karls Universit\"at T\"ubingen\\
Mathematisches Institut\\
Auf der Morgenstelle 10\\
72076 T\"ubingen, Germany\\
{\verb+martin.zerner@uni-tuebingen.de+} 

\end{document}